\documentclass[a4paper,leqno]{amsart}
\usepackage[margin=30truemm]{geometry}
\usepackage{fancyhdr,appendix,graphicx,layout}
\usepackage{stix2}
\usepackage{amsmath,amsthm}
\usepackage{enumitem}

\usepackage[numbers]{natbib}
\usepackage{bbm}
\usepackage{bm}
\usepackage{tikz}
\usepackage{tikz-cd}
\usetikzlibrary{positioning,arrows.meta,decorations.markings}
\tikzset{
    cells={font=\everymath\expandafter{\the\everymath\displaystyle}},
}
\usepackage[pagewise]{lineno}

\AtBeginEnvironment{subappandices}{%
\chapter*{Appendix}
\addcontentsline{toc}{chapter}{Appendices}
\counterwithin{figure}{section}
\counterwithin{table}{section}
}

\renewcommand{\emph}[1]{\textcolor{blue}{\textit{#1}}}

\usepackage{hyperref}
\hypersetup{
    colorlinks=true,
    linkcolor=red,  
    filecolor=blue,  
    urlcolor=orange, 
    citecolor=purple   
}
\usepackage{cleveref}

\makeatletter
\newcommand\ReDeclareMathOperator[2]{%
    \begingroup \escapechar\m@ne\xdef\@gtempa{{\string#1}}\endgroup
    \expandafter\@ifundefined\@gtempa
        {\@latex@error{Command \string#1 undefined}\@ehc}
        \relax
    \let\@ifdefinable\@rc@ifdefinable
    \DeclareMathOperator#1{#2}}
\makeatother

\crefname{thm}{Theorem}{Theorems}
\crefname{dfn}{Definition}{Definitions}
\crefname{dfnprop}{Definition-Proposition}{Definition-Proposition}
\crefname{prop}{Proposition}{Propositions}
\crefname{lem}{Lemma}{Lemmas}
\crefname{cor}{Corollary}{Corollaries}
\crefname{clm}{Claim}{Claims}
\crefname{ass}{Assumption}{Assumption}
\crefname{cond}{Condition}{Condition}
\crefname{nota}{Notation}{NOtation}
\crefname{conj}{Conjecture}{Conjecture}
\crefname{fct}{Fact}{Facts}
\crefname{rmk}{Remark}{Remarks}
\crefname{eg}{Example}{Examples}
\crefname{que}{Question}{Questions}
\crefname{figure}{Figure}{Figures}
\crefname{table}{Table}{Tables}
\crefname{section}{Section}{Sections}
\crefname{subsection}{Subsection}{Subsections}
\crefname{appendix}{Appendix}{Appendices}
\crefname{equation}{}{}
\crefname{main}{Theorem}{Theorems}

\theoremstyle{definition}
\newtheorem{thm}{Theorem}[section]
\newtheorem{main}{Theorem}
\newtheorem{dfn}[thm]{Definition}

\newtheorem{prop}[thm]{Proposition}
\newtheorem{lem}[thm]{Lemma}
\newtheorem{cor}[thm]{Corollary}

\newtheorem{rmk}[thm]{Remark}
\newtheorem{eg}[thm]{Example}

\numberwithin{equation}{section}
\makeatletter
\let\c@equation\c@thm
\makeatother

\newcommand{\D}{\mathcal{D}}

\newcommand{\K}{\mathcal K}
\renewcommand{\L}{\mathcal L}

\newcommand{\T}{\mathcal T}

\newcommand{\RR}{\mathbf{R}}

\newcommand{\NN}{\mathbb N}

\newcommand{\ZZ}{\mathbb{Z}}

\newcommand{\xto}{\xrightarrow}
\newcommand{\LR}{\Leftrightarrow}
\newcommand{\To}{\Rightarrow}

\newcommand{\ts}[1]{\overline{#1}}

\DeclareMathOperator{\add}{add}

\ReDeclareMathOperator{\top}{top}

\DeclareMathOperator{\ind}{ind}

\DeclareMathOperator{\obj}{obj}

\let\mod\relax
\DeclareMathOperator{\mod}{mod}

\newcommand{\op}{\text{op}}

\DeclareMathOperator{\Hom}{Hom}

\DeclareMathOperator{\End}{End}

\DeclareMathOperator{\RHom}{\RR Hom}

\DeclareMathOperator{\REnd}{\RR End}

\DeclareMathOperator{\Tot}{\mathsf{Tot}}

\DeclareMathOperator{\thick}{thick}

\DeclareMathOperator{\Ext}{Ext}
\DeclareMathOperator{\silt}{silt}

\DeclareMathOperator{\cone}{cone}

\DeclareMathOperator{\per}{per}
\DeclareMathOperator{\pvd}{pvd}

\newcommand{\wt}{\widetilde}
\title{Silting-discrete graded path algebras}
\author{Riku Fushimi}
\date{\today}

\newcommand{\Addresses}{{
  \bigskip
  \footnotesize

  R. Fushimi, \textsc{Department of mathematics, Nagoya University,
  Chikusa-ku, Nagoya 464-8602, Japan}\par\nopagebreak
  \textit{E-mail address}: \texttt{fushimi.riku.h9@s.mail.nagoya-u.ac.jp}
}}

\begin{document}

\begin{abstract}
  We classify connected finite acyclic graded quivers $Q$ for which the
  graded path algebra $kQ$, regarded as a formal dg algebra, is
  silting-discrete.  We prove that $kQ$ is silting-discrete if and only if
  it is derived-discrete, and that both conditions are equivalent to the
  underlying graph of $Q$ being of type ADE, or of type $\wt{A}$ with
  unequal clockwise and counter-clockwise total degrees.  The key
  ingredient is an explicit construction of an infinite pre-simple-minded
  collection in $\pvd kQ$ in the non-discrete case.
\end{abstract}

\maketitle
\tableofcontents

\section{Introduction}\label{section:Intro}

Gabriel's theorem asserts that the path algebra $kQ$ of a connected finite
acyclic quiver $Q$ over a field $k$ admits only finitely many isomorphism
classes of indecomposable modules if and only if the underlying graph of
$Q$ is a Dynkin diagram of type ADE.  This celebrated result is one of the
foundational theorems of representation theory: it shows that a finiteness
condition on the module category is controlled entirely by the
combinatorial shape of the quiver.  Gabriel's proof is constructive: when
$Q$ is not Dynkin, an infinite family of pairwise non-isomorphic
indecomposable representations is exhibited explicitly.

Several weakenings of representation-finiteness have since been studied at
the level of derived categories.  Vossieck \cite{V01} introduced
\emph{derived-discrete} algebras: an algebra $\Lambda$ is
derived-discrete if for every cohomology dimension vector, $\D^b(\mod\Lambda)$
contains only finitely many isomorphism classes of objects with that
vector.  He classified all such algebras over algebraically closed fields:
a connected derived-discrete algebra is either piecewise hereditary of
Dynkin type, or a gentle one-cycle algebra satisfying a certain
combinatorial condition (the absence of a clock). 

Keller-Vossieck
\cite{KV88} introduced \emph{silting objects} as a generalization of
tilting objects, and Aihara-Iyama \cite{AI12} developed the mutation theory of silting objects.  An algebra $\Lambda$ is \emph{silting-discrete} if for every
$l\ge 0$, the set of silting objects lying between $\Lambda$ and
$\Sigma^l\Lambda$ in the silting partial order is finite. For an ordinary quiver $Q$,
\begin{align*}
    Q\text{ is Dynkin }\LR kQ\text{ is derived-discrete }\LR kQ\text{ is silting-discrete.}
\end{align*}

In this paper, we study silting-discreteness for \emph{graded} path
algebras.  A graded quiver $Q$ is a quiver in which each arrow carries an
integer degree, and the path algebra $kQ$ then inherits a grading by total
degree of paths; viewed with zero differential, $kQ$ is a formal dg
algebra.  This class strictly contains ordinary path algebras.  Our main result gives a
complete classification.

\begin{main}\label{main:thm}
  Let $Q$ be a connected finite acyclic graded quiver.  The following
  conditions are equivalent:
  \begin{itemize}
    \item[(1)] The underlying graph of $Q$ is of type ADE, or of type
      $\wt{A}$ with unequal clockwise and counter-clockwise total degrees;
    \item[(2)] $\#\ind\left(\pvd^{[-n,0]}kQ\right)<\infty$ for every $n\ge 0$;
    \item[(3)] $kQ$ is silting-discrete.
  \end{itemize}
\end{main}

The principal difficulty lies in proving $(3)\To(1)$.  By a
field-extension argument (\cref{lem:base-change}), we may assume that $k$
is uncountable.  We prove the contrapositive: if condition~(1) fails, we
construct a pre-simple-minded collection $\{L_\lambda\}_{\lambda\in k^*}
\subseteq\pvd kQ$ indexed by $k^*$, and apply a result of Hara-Wemyss
\cite{HW26}, which bounds the cardinality of any pre-simple-minded
collection in $\pvd kQ$ by $\#Q_0$ whenever $kQ$ is silting-discrete.
The existence of such a collection in the non-discrete case is established
in \cref{section:psmc}, where three explicit families of graded quivers
(\cref{prop:piece}) are treated directly; successive applications of vertex
deletion, vertex contraction, and sink/source mutation
(\cref{subsec:reduction}) reduce the general case to these three families.

\medskip\noindent
\textbf{Organization.}
\Cref{section:Pre} collects the prerequisites.
\Cref{subsec:silting} recalls silting theory and introduces
pre-simple-minded collections.
\Cref{subsec:graded} reviews graded algebras and their relation to formal
dg algebras.
\Cref{subsec:reduction} introduces the three reduction operations on
graded quivers and verifies that each preserves silting-discreteness.
\Cref{section:main} proves \cref{main:thm}, deferring the explicit
construction of pre-simple-minded collections to \cref{section:psmc}.

\medskip\noindent
\textbf{Conventions and notation.}
We fix a field $k$. All modules are right modules.  All subcategories are
full.
For a graded quiver $Q$, we write $Q_0$ for the set of vertices and $Q_1$
for the set of arrows, and use $s(\alpha)$, $t(\alpha)$, $\deg(\alpha)$
for the source, target, and degree of an arrow $\alpha$.

\medskip\noindent
\textbf{Acknowledgements.}
The author would like to express his sincere gratitude to his supervisor
Akira Ishii for his continuous support and encouragement.  The author also
thanks Osamu Iyama for suggesting the problem of classifying
silting-discrete graded path algebras.

\section{Preliminaries}\label{section:Pre}

\subsection{Silting theory}\label{subsec:silting}

Let $\T$ be a Hom-finite Krull-Schmidt triangulated category with shift functor $\Sigma$.

\begin{dfn}[\cite{KV88,AI12}]
  An object $M\in\T$ is a \emph{silting object} if
  \begin{itemize}
    \item $\Hom_\T(M,\Sigma^{>0}M)=0$, and
    \item $\thick M = \T$.
  \end{itemize}
  Two silting objects are \emph{equivalent} if they have the same
  additive closure; we write $\silt\T$ for the set of equivalence classes.
\end{dfn}

The set $\silt\T$ carries a natural partial order: $M\ge N$ if
$\Hom_\T(M,\Sigma^{>0}N)=0$ \cite{AI12}.  \emph{Silting mutation}
\cite{AI12} replaces one indecomposable summand of a basic silting object
via an approximation triangle to produce a new basic silting object; two
basic silting objects are related by a mutation if and only if they are
adjacent in the Hasse quiver of $(\silt\T,\ge)$.

\begin{dfn}
  Let $T_0\in\silt\T$.  A silting object $T\in\silt\T$ is
  \emph{reachable from $T_0$} if there is a finite sequence of silting
  mutations connecting $T_0$ to $T$.  We say $\T$ is
  \emph{silting-discrete} if for every $l\ge 0$ the
  set
  \[
    \bigl\{T\in\silt\T \;\big|\; T_0 \ge T \ge \Sigma^l T_0\bigr\}
  \]
  is finite. Silting-discrete does not depend on the choice of $T_0$.

\end{dfn}

\begin{dfn}
    Let $A$ be a dg algebra.
    \begin{itemize}
        \item $\per(A):=\thick_{\D(A)}A_A$ is called the \emph{perfect derived category} of $A$.
        \item $\pvd(A):=\{X\in\D(A)\mid \sum_{i\in\ZZ}\dim H^i(X)<\infty\}$ is called the \emph{perfectly-valued derived category} of $A$.
    \end{itemize}
\end{dfn}
For the definition of dg algebras and their derived categories, we refer the reader to \cite{K94}. For a subset $I\subseteq\ZZ$, we put
\begin{align*}
    \pvd^I(A):=\{X\in\pvd(A)\mid H^j(X)=0\text{ if }j\notin I\}.
\end{align*}

In this paper, we consider finite-dimensional graded path algebras $kQ$, and in this case $\per kQ=\pvd kQ$. For a dg algebra $A$, we say  $A$ is \emph{silting-discrete} if $\per A$ is silting-discrete.

We use pre-simple-minded collections to prove that certain graded path
algebras are not silting-discrete.

\begin{dfn}
  A family $\L=\{L_i\}_{i\in I}\subseteq\T$ is a
  \emph{pre-simple-minded collection} if
  \begin{itemize}
    \item $\Hom_\T(\L,\Sigma^{<0}\L)=0$, and
    \item $\L$ is a \emph{semibrick}: for every $i,j\in I$,
      \[
        \Hom_\T(L_i,L_j)=
        \begin{cases}
          \text{division ring} & \text{if }i=j,\\
          0                   & \text{if }i\ne j.
        \end{cases}
      \]
  \end{itemize}
  A pre-simple-minded collection $\L$ is called \emph{simple-minded collection} if $\thick\L=\T$.
\end{dfn}

The relevance to silting-discreteness comes from the following result.

\begin{lem}\label{lem:HW}
    Let $A$ be a connective dg algebra with finite-dimensional cohomologies. If $A$ is silting-discrete, then any pre-simple-minded collection can be completed to a simple-minded collection. In particular, every pre-simple-minded collection is finite.
\end{lem}
\begin{proof}
    This claim is proved in \cite{HW26} for finite-dimensional algebras, and their proof carries over to the general case.
\end{proof}

\begin{cor}\label{cor:not-sd}
  Assume $k$ is uncountable.  If $\pvd A$ contains a pre-simple-minded
  collection $\{L_\lambda\}_{\lambda\in k^*}$ indexed by $k^*$, then $A$
  is not silting-discrete.
\end{cor}

\subsection{Graded algebras and formal dg algebras}%
\label{subsec:graded}

\begin{dfn}
  Let $A=\bigoplus_{i\in\ZZ}A^i$ be a finite-dimensional graded algebra.
  We denote by $\mod^\ZZ A$ the category of graded $A$-modules.  For
  $M,N\in\mod^\ZZ A$, we write
  \[
    \Hom_A^\ZZ(M,N)
      \;:=\; \Hom_{\mod^\ZZ A}(M,N)
      \;=\; \Hom_A(M,N)^0.
  \]
  We define $M(i)\in\mod^\ZZ A$ by $M(i)^j:=M^{i+j}$.
\end{dfn}

In the rest of this paper, we mainly consider negatively graded algebras.
If $A$ is negatively graded, then for every $M\in\mod^\ZZ A$ the subspace
$M^{\le i}:=\bigoplus_{n\le i}M^n$ is a graded submodule, and we write
$M^{\ge i}$ for the quotient $M/M^{\le i-1}$.  For a subset $I\subseteq\ZZ$
we put
\begin{align*}
  \mod^I A \;:=\; \{M\in\mod^\ZZ A\mid M^j=0\text{ for every }j\notin I\}.
\end{align*}
For a graded quiver $Q$, the path algebra $kQ$ has a natural grading by total degree of paths, and every graded algebra may be regarded as a formal dg algebra (i.e.\ a dg algebra with zero differential).  The following theorem of Kalck-Yang identifies the perfectly-valued derived category of $kQ$, viewed as a formal dg algebra, with an orbit category of the bounded derived category of graded modules.

\begin{thm}[\cite{KY18}, Theorem~1.3]\label{thm:base}
  Let $Q$ be an acyclic graded quiver.  The functor taking total complexes
  induces an equivalence
  \begin{align*}
    \Tot\colon
    \D^b(\mod^\ZZ kQ)\,/\,\Sigma(-1)
    \;\overset{\sim}{\longrightarrow}\;
    \pvd kQ.
  \end{align*}
  In particular, it induces a bijection $\ind(\mod^{[-n,0]}kQ)\simeq\ind(\pvd^{[-n,0]}kQ)$.
\end{thm}

\begin{lem}\label{lem:base-change}
  Let $k\subseteq l$ be an extension of fields.  The functor
  $(-)_l\colon\per kQ\to\per lQ$ induces a bijection between the
  isomorphism classes of reachable silting complexes.  In particular, $kQ$
  is silting-discrete if and only if so is $lQ$.
\end{lem}

\begin{proof}
  By \cite[Remark~7.7]{KY14}, for every reachable basic silting object
  $M=\bigoplus_{i=1}^n M_i\in\per kQ$ we have
  $\top\End_{\per kQ}(M)\simeq k^n$, so $(M_i)_l$ is indecomposable.
  Since left approximations are preserved under $(-)_l$, the functor
  induces a bijection between the isomorphism classes of reachable silting
  complexes.  The last claim then follows from
  \cite[Theorem~1.2]{AM17} and \cite[Corollary~2.38]{AIR14}.
\end{proof}

\subsection{Reduction operations on graded quivers}\label{subsec:reduction}

\begin{dfn}[vertex contraction, vertex deletion]
  Let $i$ be a vertex of $Q$.  We define graded quivers $Q(i)$ and
  $Q(\hat i)$ as follows:
  \begin{itemize}
    \item $Q(i)_0=Q(\hat i)_0=Q_0\setminus\{i\}$;
    \item $Q(\hat i)_1
          =\{\alpha\in Q_1\mid s(\alpha)\ne i\ne t(\alpha)\}$
          with natural gradings;
    \item $Q(i)_1=Q(\hat i)_1\sqcup
          \{\alpha\beta\mid s(\alpha)=i=t(\beta)\}$
          with natural gradings.
  \end{itemize}
\end{dfn}

\begin{eg}
    Consider a quiver $Q=$\begin{tikzcd}
        & 2\ar[rd,"-1"] & \\ 1\ar[rr]\ar[ru] & & 3.
    \end{tikzcd}
    Then we can compute that $Q(2)=$\begin{tikzcd}
        1\ar[r,yshift=0.7ex,,"-1"]\ar[r,yshift=-0.7ex] & 2
    \end{tikzcd} and $Q(\hat 2)=$\begin{tikzcd}
        1\ar[r] & 2.
    \end{tikzcd}
\end{eg}

\begin{lem}\label{lem:reduce}
  Let $i\in Q_0$.  If $kQ$ is silting-discrete, then so are $kQ(\hat i)$
  and $kQ(i)$.
\end{lem}

\begin{proof}
  We have equivalences
  $\per kQ(i)\simeq\thick(1-e_i)kQ$ and
  $\per kQ(\hat i)\simeq\per kQ/\thick e_ikQ$.
  The claim follows from \cite[Theorem~2.10 and Lemma~2.12]{AH24}.
\end{proof}

\begin{dfn}\label{dfn:sink}
  A vertex $i\in Q_0$ is a \emph{sink} if
  \begin{itemize}
      \item there are no arrows with source $i$,
      \item all arrows with target $i$ carry the same degree.
  \end{itemize}
  \emph{Source} vertices are defined analogously.
\end{dfn}

\begin{rmk}
    The notion of sink in \cref{dfn:sink} is more restrictive than the usual one. For example, consider a graded quiver $Q=$ \begin{tikzcd}
            1\ar[r,"0"] & 2 & 3\ar[l,"-1"'].
        \end{tikzcd}
    The vertex $2$ is a sink in the usual sense but not in the sense of \cref{dfn:sink}, since not all arrows into $2$ carry the same degree.
\end{rmk}

Let $i\in Q_0$ be a sink vertex, and suppose every arrow $\alpha$ with
$t(\alpha)=i$ has the same degree $r$.  Then there is a quasi-isomorphism
\begin{align*}
  k(\mu_i Q)\;\simeq\;
  \REnd_{kQ}\!\left(
    \bigoplus_{j\ne i}e_j kQ
    \;\oplus\;
    \Sigma^{-1}\cone\!\left[
      \Sigma^{-r}\!\!\bigoplus_{t(\alpha)=i}\!\!e_{s(\alpha)}kQ
      \longrightarrow e_i kQ
    \right]
  \right),
\end{align*}
where $\mu_i Q$ is the graded quiver with $(\mu_i Q)_0=Q_0$ and
\begin{align*}
  (\mu_i Q)_1
    = \{\alpha\mid t(\alpha)\ne i\}
    \sqcup \{\alpha^*\mid t(\alpha)=i\},
\end{align*}
and $s(\alpha^*)=i$, $t(\alpha^*)=s(\alpha)$, $\deg(\alpha^*)=-r$.
We call $\mu_i Q$ the \emph{sink mutation} of $Q$ at $i$; \emph{source mutations}
are defined dually.

\begin{eg}
    Consider a quiver $Q=$ \begin{tikzcd}
        & 2 & \\
        1\ar[ru]\ar[rr,"-1"'] & & 3\ar[lu].
    \end{tikzcd}
    Then the mutation of $Q$ at the vertex $2$ is  $\mu_2Q=$\begin{tikzcd}
        & 2\ar[ld]\ar[rd] & \\
        1\ar[rr,"-1"'] & & 3.
    \end{tikzcd}
\end{eg}

\begin{lem}\label{lem:mutation}
  Let $i\in Q_0$ be a sink vertex.  Then $kQ$ is silting-discrete if and
  only if so is $k(\mu_i Q)$.
\end{lem}

\begin{proof}
  The formal dg algebras $kQ$ and $k(\mu_i Q)$ are derived equivalent.
\end{proof}

\begin{rmk}
  It is clear that $kQ$ is silting-discrete if and only if so is $kQ^\op$.
\end{rmk}

\section{Main theorem}\label{section:main}

\begin{dfn}
  Let $n\ge 0$.  Define a quiver $\wt{Q}^{[-n,0]}$ by
  \begin{itemize}
    \item $\wt{Q}^{[-n,0]}_0
          =\{(i,l)\mid i\in Q_0,\;l\in[-n,0]\}$;
    \item $\wt{Q}^{[-n,0]}_1
          =\bigl\{(i,l)\xto{\alpha_l}(j,l+\deg\alpha)
            \;\big|\;
            i\xto{\alpha}j,\;
            \{l,l+\deg\alpha\}\subseteq[-n,0]\bigr\}$.
  \end{itemize}
\end{dfn}

\begin{eg}
  If $Q=$
  \begin{tikzcd}
    1\ar[r,yshift=0.7ex]\ar[r,yshift=-0.7ex,"-1"'] & 2\ar[r] & 3
  \end{tikzcd},
  then $\wt{Q}^{[-2,0]}=$
  \begin{tikzcd}[row sep=0.5em]
    (1,0)\ar[r]\ar[rd]  & (2,0)\ar[r]  & (3,0)  \\
    (1,-1)\ar[r]\ar[rd] & (2,-1)\ar[r] & (3,-1) \\
    (1,-2)\ar[r]        & (2,-2)\ar[r] & (3,-2)
  \end{tikzcd}.
  Since $\wt{Q}^{[-2,0]}$ is not Dynkin, $\#\ind\left(\mod k\wt{Q}^{[-2,0]}\right)=\infty$.
\end{eg}

\begin{rmk}\label{rmk:equiv}
    Assume that $Q$ is non-positively graded.
    The category $\mod^{[-n,0]}kQ$ has a progenerator $\bigoplus_{i=0}^n kQ(i)^{\ge-n}$ whose endomorphism ring is $k\wt{Q}^{[-n,0]}$, giving an equivalence
    \[
        \mod^{[-n,0]}kQ \;\simeq\; \mod k\wt{Q}^{[-n,0]}.
    \]
\end{rmk}

The proof of the following proposition is given in \cref{section:psmc}.

\begin{prop}\label{prop:piece}
  Suppose $Q$ has one of the following forms:
  \begin{itemize}
    \item[(a)]
      \begin{tikzcd}
        1\ar[r,yshift=0.7ex]\ar[r,yshift=-0.7ex,"-n"'] & 2\ar[r] & 3
      \end{tikzcd}
      \quad($n>0$);
    \item[(b)]
      \begin{tikzcd}
        1\ar[r,yshift=0.7ex]\ar[r,yshift=-0.7ex,"-n"'] & 2 &
        3\ar[l,yshift=-0.7ex,"-m"]\ar[l,yshift=0.7ex]
      \end{tikzcd}
      \quad($n,m>0$);
    \item[(c)] the $n$-Kronecker quiver \quad($n\ge 3$).
  \end{itemize}
  Then there is a pre-simple-minded collection
  $\{L_\lambda\}_{\lambda\in k^*}\subseteq\pvd kQ$.
  In particular, if $\#k$ is infinite, then $kQ$ is not silting-discrete.
\end{prop}

\begin{thm}\label{thm:main}
  Let $Q$ be a connected acyclic graded quiver.  The following conditions
  are equivalent:
  \begin{itemize}
    \item[(1)] The underlying graph of $Q$ is of type ADE, or of type
      $\wt{A}$ with unequal clockwise and counter-clockwise total
      degrees;
    \item[(2)] $\#\ind\left(\pvd^{[-n,0]}kQ\right)<\infty$ for every $n\in\NN$;
    \item[(3)] $kQ$ is silting-discrete.
  \end{itemize}
\end{thm}

\begin{proof}
  Since $Q$ is connected and acyclic, after a suitable degree shift we may
  assume that $Q$ is non-positively graded and that the degree-$0$ part
  $Q^0$ is connected.  Under this assumption, condition~(1) is equivalent
  to
  \begin{itemize}
    \item[(1)$'$] The underlying graph of $Q$ is of type ADE, or of type
      $\wt{A}$ with only one negative arrow.
  \end{itemize}

  \noindent
  $(1)'\To(2)$:
  The hypothesis implies that $\wt{Q}^{[-n,0]}$ is a disjoint union of Dynkin quivers for every $n$.  By \cref{rmk:equiv} and \cref{thm:base}, $\#\ind(\pvd^{[-n,0]}kQ)=\#\ind(\mod^{[-n,0]}kQ)=\#\ind(\mod k\wt{Q}^{[-n,0]})$, which is finite.

  \noindent
  $(2)\To(3)$: Immediate from the definitions.

  \noindent
  $(3)\To(1)'$:
  By \cref{lem:base-change} we may assume $k=\ts{k}$.  Since $kQ^0$ is
  $\tau$-tilting finite, $Q^0$ is of type ADE.  We assume
  \begin{itemize}[noitemsep]
    \item[$(*)$] $Q^{<0}\ne\emptyset$ and the underlying graph of $Q$ is
      not of type $\wt{A}_n$ for any $n$,
  \end{itemize}
  and derive a contradiction with silting-discreteness.

  Because $Q^0$ is a connected tree, any two vertices of $Q$ are connected
  by a unique undirected walk in $Q^0$; write $d_{Q^0}(i,j)$ for the
  length of this walk.  Choose $\alpha\in Q^{<0}$ such that
  \begin{itemize}[noitemsep]
    \item[(i)] $n:=d_{Q^0}(s(\alpha),t(\alpha))$ is minimal among all
      $\alpha'\in Q^{<0}$;
    \item[(ii)] subject to~(i), $N:=\#\{\beta\in Q^{<0}\mid
      s(\beta)=s(\alpha),\;t(\beta)=t(\alpha)\}$ is maximal.
  \end{itemize}
  Let $W\subseteq Q^0$ denote the unique walk from $s(\alpha)$ to
  $t(\alpha)$.  By \cref{lem:reduce} and \cref{lem:mutation}, the three
  operations
  \[
    Q\mapsto Q(\hat j),
    \qquad
    Q\mapsto Q(j),
    \qquad
    Q\mapsto\mu_j Q
  \]
  each preserve silting-discreteness when applicable.  We apply them
  successively to produce a quiver matching one of the forms (a), (b), (c)
  in \cref{prop:piece}, contradicting silting-discreteness.

  \medskip
  \textsc{Case }$N\ge 2$:
  Applying $Q\mapsto Q(\hat j)$ for every $j\notin W$, we may assume
  $Q^0$ is of type $A_{n+1}$ with endpoints $s(\alpha)$ and $t(\alpha)$.
  If $n=1$, then $Q$ is the $(N+1)$-Kronecker quiver, matching~(c).
  Assume $n>1$.  We iteratively reduce $n$ as follows.  While $n\ge 2$:
  \begin{itemize}
    \item If some vertex $i$ has both an incoming and an outgoing arrow,
      apply $Q\mapsto Q(i)$; this removes $i$ and decreases $n$ by $1$.
    \item Otherwise every interior vertex is a sink or a source.  Pick any
      such $i$, apply $Q\mapsto\mu_i Q$; the mutation flips the arrows at
      $i$, so its neighbours now have both incoming and outgoing arrows.
      Return to the first step.
  \end{itemize}
  The procedure terminates at $n=1$, giving the $(N+1)$-Kronecker quiver.

  \medskip
  \textsc{Case }$N=1$:
  Since the underlying graph of $Q$ is not of type $\wt{A}$, there exists
  a vertex $i_0\notin W$ adjacent to $W$.  Applying $Q\mapsto Q(\hat j)$
  to every $j\notin W\sqcup\{i_0\}$, we may assume $Q_0=W\sqcup\{i_0\}$.
  The possible shapes of $Q$ are exactly the following:
  \begin{itemize}
    \item[i)]
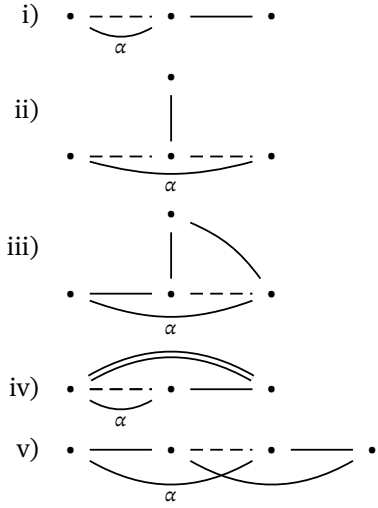
      \begin{tikzcd}
        \bullet\ar[r,bend right,"\alpha"',no head]\ar[r,dashed,no head]
        & \bullet\ar[r,no head] & \bullet
      \end{tikzcd}
    \item[ii)]
      \begin{tikzcd}
        & \bullet\ar[d,no head] & \\
        \bullet\ar[r,no head,dashed]
        \ar[rr,bend right=15,"\alpha"',no head]
        & \bullet\ar[r,dashed,no head] & \bullet
      \end{tikzcd}
    \item[iii)]
      \begin{tikzcd}
        & \bullet\ar[d,no head]\ar[rd,bend left=15,no head] & \\
        \bullet\ar[r,no head]
        \ar[rr,bend right=20,"\alpha"',no head]
        & \bullet\ar[r,no head,dashed] & \bullet
      \end{tikzcd}
    \item[iv)]
      \begin{tikzcd}
        \bullet\ar[r,no head,dashed]
        \ar[r,bend right,"\alpha"',no head]
        \ar[rr,bend left,equal]
        \ar[r,dashed,no head]
        & \bullet\ar[r,no head] & \bullet
      \end{tikzcd}
    \item[v)]
      \begin{tikzcd}
        \bullet\ar[r,no head]
        \ar[rr,bend right,"\alpha"',no head]
        & \bullet\ar[r,dashed,no head]
        \ar[rr,bend right,no head]
        & \bullet\ar[r,no head] & \bullet
      \end{tikzcd}
  \end{itemize}
  where solid edges are arrows of $Q$, dashed edges represent (unoriented) paths in $Q$, curved edges represent negative arrows, and curved double edges represent possibly multiple negative arrows.  The same argument as in
  the case $n=2$ reduces each of~i)--v) to one of~(a), (b), (c).
\end{proof}

We illustrate the reduction procedure with concrete examples.

\begin{eg}
    Consider a quiver $Q=$ 
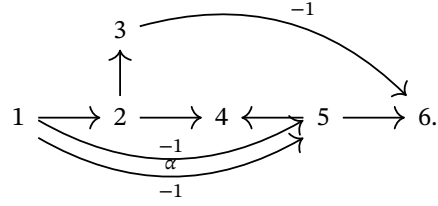
\begin{tikzcd}
        & 3\ar[rrrd,bend left,"-1"] & & & \\
        1\ar[r]\ar[rrr,bend right,yshift=-0.7ex,"-1"',"\alpha"]\ar[rrr,bend right,yshift=0.7ex,"-1"] & 2\ar[u]\ar[r] & 4 & 5\ar[l]\ar[r] & 6.
    \end{tikzcd}
    \begin{itemize}
        \item $Q\mapsto Q(\hat 3)$;
        \begin{tikzcd}
            1\ar[r]\ar[rrr,bend right,yshift=-0.7ex,"-1"']\ar[rrr,bend right,"-1"] & 2\ar[r] & 4 & 5\ar[l]\ar[r] & 6.
        \end{tikzcd}
        \item $Q\mapsto Q(\hat 6)$;
        \begin{tikzcd}
            1\ar[r]\ar[rrr,bend right,yshift=-0.7ex,"-1"']\ar[rrr,bend right,"-1"] & 2\ar[r] & 4 & 5\ar[l].
        \end{tikzcd}
        \item $Q\mapsto Q(2)$;
        \begin{tikzcd}
            1\ar[r]\ar[rr,bend right,yshift=-0.7ex,"-1"']\ar[rr,bend right,"-1"] & 4 & 5\ar[l].
        \end{tikzcd}
        \item $Q\mapsto Q(5)$;
        \begin{tikzcd}
            1\ar[r]\ar[r]\ar[r,bend right,yshift=-0.7ex,"-1"']\ar[r,bend right,"-1"] & 5.
        \end{tikzcd}
    \end{itemize}
    The resulting quiver is $(c)$ in \cref{prop:piece}.
\end{eg}

\begin{eg}
    Consider a quiver $Q=$ \begin{tikzcd}
        & & & 4 & \\
        0\ar[rrru, bend left,"-1"] & 1\ar[l]\ar[rrr, bend right,"\alpha","-1"'] & 2\ar[l] & 3\ar[l]\ar[u]\ar[r] & 5 .
    \end{tikzcd}
    \begin{itemize}
        \item $Q\mapsto Q(\hat 0)$; \begin{tikzcd}
            & & 4 & \\
            1\ar[rrr, bend right,"-1"'] & 2\ar[l] & 3\ar[l]\ar[u]\ar[r] & 5 .
        \end{tikzcd}
        \item $Q\mapsto Q(1)$; \begin{tikzcd}
            & 4 & \\
            2\ar[rr, bend right,"-1"'] & 3\ar[l]\ar[u]\ar[r] & 5 .
        \end{tikzcd}
        \item $Q\mapsto Q(2)$; 
        \begin{tikzcd}
            4 & 3\ar[l]\ar[r, bend right,"-1"']\ar[r] & 5 .
        \end{tikzcd}
        \item $Q\mapsto\mu_4Q$;
        \begin{tikzcd}
            4\ar[r] & 3\ar[r, bend right,"-1"']\ar[r] & 5 .
        \end{tikzcd}
    \end{itemize}
    The resulting quiver is opposite to $(a)$ in \cref{prop:piece}.
\end{eg}

\begin{eg}
    Consider a quiver $Q=$ \begin{tikzcd}
        1\ar[r]\ar[rrrr,bend right, "-1"'] & 2\ar[r] & 3 & 4\ar[l]\ar[r] & 5 & 6\ar[l]\ar[llll,bend left,"-1"].
    \end{tikzcd}
    \begin{itemize}
        \item $Q\mapsto Q(2)$; 
        \begin{tikzcd}
            1\ar[r]\ar[rrr,bend right, "-1"'] & 3 & 4\ar[l]\ar[r] & 5 & 6\ar[l]\ar[lll,bend left,"-1"].
        \end{tikzcd}
        \item $Q\mapsto \mu_4Q$;
        \begin{tikzcd}
            1\ar[r]\ar[rrr,bend right, "-1"'] & 3\ar[r] & 4 & 5\ar[l] & 6\ar[l]\ar[lll,bend left,"-1"].
        \end{tikzcd}
        \item $Q\mapsto Q(3)$;
        \begin{tikzcd}
            1\ar[r]\ar[rr,bend right, "-1"'] & 4 & 5\ar[l] & 6\ar[l]\ar[ll,bend left,"-1"].
        \end{tikzcd}
        \item $Q\mapsto Q(5)$;
        \begin{tikzcd}
            1\ar[r]\ar[r,bend right, "-1"'] & 4 & 6\ar[l]\ar[l,bend left,"-1"].
        \end{tikzcd}
    \end{itemize}
    The resulting quiver is $(b)$ in \cref{prop:piece}.
\end{eg}

\section{Infinite pre-simple-minded collection}\label{section:psmc}

In this section, we prove \cref{prop:piece} by constructing the required pre-simple-minded collections explicitly. Throughout this section, we freely identify
$\obj(\pvd^{[-n,0]}kQ)$,
$\obj(\mod^{[-n,0]}kQ)$,
and $\obj(\mod k\wt{Q}^{[-n,0]})$
via the equivalences of \cref{thm:base} and \cref{rmk:equiv}.

\begin{lem}\label{lem:i.e.}
  Let $\L=\{L_i\}_{i\in I}\subseteq\mod kQ$.  The following are equivalent:
  \begin{itemize}
    \item[(1)] $\L$ is a pre-simple-minded collection in $\pvd kQ$;
    \item[(2)] $\L$ is a semibrick in $\mod kQ$, and
      $\Hom_{kQ}(L,L)^{<0}=0=\Ext^1_{kQ}(L,L)^{<0}$.
  \end{itemize}
\end{lem}

\begin{proof}
  By \cref{thm:base} there are isomorphisms
  \[
    \Hom_{\pvd kQ}(L,\Sigma^i L)
    \;\simeq\;
    \bigoplus_{s+t=i}\Ext^s_{\mod kQ}(L,L(t))
    \;\simeq\;
    \Hom_{kQ}(L,L)^i \oplus \Ext^1_{kQ}(L,L)^{i-1},
  \]
  where the second isomorphism uses the fact that $kQ$ is hereditary.
  Hence
  \begin{itemize}
    \item $\Hom_{\pvd kQ}(L_i,L_j)=\delta_{ij}k$ iff
      $\Hom^\ZZ_{kQ}(L_i,L_j)=\delta_{ij}k$ and
      $\Ext^1_{kQ}(L_i,L_j)^{-1}=0$;
    \item $\Hom_{\pvd kQ}(L_i,\Sigma^{<0}L_j)=0$ iff
      $\Hom_{kQ}(L_i,L_j)^{<0}=0=\Ext^1_{kQ}(L_i,L_j)^{<-1}$.
      \qedhere
  \end{itemize}
\end{proof}

\begin{lem}\label{lem:trivial}
  Let $L_i,L_j\in\mod kQ$ and let $0\to P_1\to P_0\to L_i\to 0$ be a
  projective resolution of $L_i$.  If $\Hom_{kQ}(P_0,L_j)^{<0}\simeq
  \Hom_{kQ}(P_1,L_j)^{<0}$ as graded vector spaces, then
  $\Hom_{kQ}(L_i,L_j)^{<0}=0$ implies $\Ext^1_{kQ}(L_i,L_j)^{<0}=0$.
\end{lem}

We prove \cref{prop:piece}(a).  It suffices to treat the case $n=1$.

\begin{prop}\label{prop:case-a}
  Let $Q=\begin{tikzcd}
    1\ar[r,yshift=0.7ex]\ar[r,yshift=-0.7ex,"-1"'] & 2\ar[r] & 3
  \end{tikzcd}$,
  and set $L_\lambda=$
  \begin{tikzcd}[row sep=0.5cm,column sep=large]
    k\ar[r]\ar[rd,"{(1\ \lambda)^T}"near end] & k & \\
    k^2\ar[r]\ar[rd] & k^2\ar[r,"{(1\ 0)}"] & k \\
    k^2\ar[r]\ar[rd,"{(0\ 1)}"near end] & k^2\ar[r,"{(1\ {-1})}"] & k \\
    & k &
  \end{tikzcd}$\in\mod^{[-3,0]}kQ^\op$.
  Then $\{L_\lambda\}_{\lambda\in k^*}\subseteq\pvd kQ^\op$ is a
  pre-simple-minded collection.
\end{prop}

\begin{proof}
  One checks directly that $\Hom^\ZZ_{kQ^\op}(L_\lambda,L_\mu)=\delta_{\lambda\mu}k$.
  We show $\Hom_{kQ^\op}(L_\lambda,L_\lambda)^{<0}=0$.  Let
  $f\in\Hom^\ZZ_{kQ^\op}(L_\lambda,L_\lambda(-1))$.  We have
  $f_{(2,-3)}=f_{(1,-2)}=f_{(3,-2)}=f_{(3,0)}=0$.  Thus
  $f_{(1,-1)}=f_{(2,-1)}$ has the form $\begin{pmatrix}a&b\\0&0\end{pmatrix}$
  and $f_{(1,0)}=f_{(2,0)}$ has the form $\begin{pmatrix}0\\c\end{pmatrix}$.
  From $\begin{pmatrix}1&-1\end{pmatrix}\begin{pmatrix}a&b\\0&0\end{pmatrix}
  =f_{(3,-1)}\begin{pmatrix}1&0\end{pmatrix}$
  we get $b=0$ and $f_{(3,-1)}=a$.  From
  $\begin{pmatrix}a&0\\0&0\end{pmatrix}\begin{pmatrix}1\\\lambda\end{pmatrix}
  =\begin{pmatrix}0\\c\end{pmatrix}$
  we get $a=0=c$.  Hence $f=0$.  One similarly checks
  $\Hom_{kQ^\op}(L_\lambda,L_\lambda)_{<-1}=0$.

  We have the projective resolution $0\to P^\lambda_1\to P^\lambda_0\to L_\lambda\to 0$ with
  \begin{itemize}
    \item $P^\lambda_0=P_1\oplus P_1(1)^2\oplus P_1(2)^2$,
    \item $P^\lambda_1=P_2(1)\oplus P_2(2)^2\oplus P_2(3)
          \oplus P_3\oplus P_3(1)\oplus P_3(2)\oplus P_3(3)$.
  \end{itemize}
  A direct computation gives
  \[
    \Hom_{kQ^\op}(P^\lambda_0,L_\mu)^{<0}
    \;=\; k(1)^6\oplus k(2)^2
    \;=\; \Hom_{kQ^\op}(P^\lambda_1,L_\mu)^{<0}.
  \]
  By \cref{lem:trivial}, $\Ext^1_{kQ^\op}(L_\lambda,L_\mu)^{<0}=0$, and by
  \cref{lem:i.e.} the collection is pre-simple-minded.
\end{proof}

We prove \cref{prop:piece}(b).  It suffices to treat $\gcd(n,m)=1$.

\begin{prop}\label{prop:case-b}
  Let $Q=\begin{tikzcd}
    1\ar[r,yshift=0.7ex]\ar[r,yshift=-0.7ex,"-m"'] & 2 &
    3\ar[l,yshift=0.7ex]\ar[l,yshift=-0.7ex,"-n"]
  \end{tikzcd}$
  with $\gcd(m,n)=1$ and $m\le n$.  Set
  \begin{center}
  $L_\lambda=$
  \begin{tikzcd}[row sep=0.1cm,column sep=2.5cm]
    k\ar[r,"\lambda"]\ar[rddd] & k & k\ar[l]\ar[ldddddd] \\
    \vdots & \vdots & \vdots \\
    k\ar[r] & k & k\ar[l]\ar[ldddddd] \\
    k\ar[r] & k & \\
    \vdots & \vdots & \\
    k\ar[r]\ar[rddd] & k\ar[luuu,leftarrow,shorten >=50pt] & \\
    k\ar[r] & k\ar[luuu,leftarrow,shorten >=50pt] & \\
    \vdots & \vdots & \\
    k\ar[r] & k
  \end{tikzcd}$\in\mod^{[-m-n+1,0]}kQ$.
  \end{center}
  Then $\{L_\lambda\}_{\lambda\in k^*}\subseteq\pvd kQ^\op$ is a pre-simple-minded
  collection.
\end{prop}

\begin{proof}
  One checks directly that $\Hom^\ZZ_{kQ^\op}(L_\lambda,L_\mu)=\delta_{\lambda\mu}k$.
  We show $\Hom_{kQ^\op}(L_\lambda,L_\lambda)^{<0}=0$.  Let
  $f\in\Hom^\ZZ_{kQ^\op}(L_\lambda,L_\lambda(-h))$ for $h>0$.  One verifies:
  \begin{itemize}
    \item $f_{(3,-l)}=0\;\To\; f_{(2,-l)}=0$ for $0\le l<m$;
    \item $f_{(2,-l)}=0\;\To\; f_{(1,-l)}=0$ for $0\le l<m+n$;
    \item $f_{(1,-l)}=0\;\To\; f_{(2,-l-m)}=0$ for $0\le l<n$;
    \item $f_{(2,-l)}=0\;\To\; f_{(3,-l+n)}=0$ for $n\le l<m+n$.
  \end{itemize}
  Since $f_{(3,-m+1)}=0$, these implications force $f=0$.  The projective
  resolution $0\to P^\lambda_1\to P^\lambda_0\to L_\lambda\to 0$ has
  \begin{itemize}
    \item $P^\lambda_0=\bigoplus_{l=0}^{m+n-1}P_1(l)\oplus\bigoplus_{l=0}^{m-1}P_3(l)$,
    \item $P^\lambda_1=\bigoplus_{l=0}^{m+n-1}P_2(l)\oplus\bigoplus_{l=n}^{m+n-1}P_2(l)$.
  \end{itemize}
  A direct computation gives
  \begin{align*}
    \Hom_{kQ^\op}(P^\lambda_0,L_\mu)^{<0}
    =\bigoplus_{l=1}^{m+n-1}k(l)^{m+n-l}\oplus\bigoplus_{l=1}^{m-1}k(l)^{m-l}
    =\Hom_{kQ^\op}(P^\lambda_1,L_\mu)^{<0}.
  \end{align*}
  \Cref{lem:trivial,lem:i.e.} then give the result.
\end{proof}

\begin{dfn}
  For a sequence of integers $a_0,a_1,\ldots,a_k$, we write
  $\K_{a_0,a_1,\ldots,a_k}$ for the $(k+1)$-Kronecker algebra whose arrows
  carry degrees $a_0,a_1,\ldots,a_k$.
\end{dfn}

We prove \cref{prop:piece}(c) first for 3-arrow Kronecker quivers.

\begin{prop}\label{prop:3-Kro}
  Let $0<m\le n$ with $\gcd(m,n)=1$.  Set
  \begin{center}
  $L_\lambda=$
  \begin{tikzcd}[row sep=0.2cm,column sep=3.0cm]
    k\ar[r]\ar[rddd]\ar[rdddddddd] & k \\
    \vdots & \vdots \\
    k\ar[r]\ar[rddd,no head,shorten >=20pt]\ar[rdddddddd] & k \\
    k\ar[r]\ar[rddd,no head,shorten >=20pt] & k \\
    \vdots\ar[rddd,shorten <=20pt] & \vdots \\
    \vdots\ar[rddd,shorten <=20pt] & \vdots \\
    \vdots & \vdots \\
    k\ar[r]\ar[rddd,"\lambda"'] & k \\
    & k \\
    & \vdots \\
    & k
  \end{tikzcd}$\in\mod^{[-m-n+1,0]}k\K_{0,-m,-n}$.
  \end{center}
  Then $\{L_\lambda\}_{\lambda\in k^*}\subseteq\pvd k\K_{0,-m,-n}^\op$ is a
  pre-simple-minded collection.
\end{prop}

\begin{proof}
  One checks directly that $\Hom^\ZZ_{k\K_{0,-m,-n}^\op}(L_\lambda,L_\mu)=\delta_{\lambda\mu}k$.
  We show $\Hom_{k\K_{0,-m,-n}^\op}(L_\lambda,L_\lambda)^{<0}=0$.  Let
  $f\in\Hom^\ZZ_{k\K_{0,-m,-n}^\op}(L_\lambda,L_\lambda(-h))$ for $h>0$.
  One verifies:
  \begin{itemize}
    \item $f_{(2,-l)}=0\;\To\; f_{(1,-l+m)}=0$ for $m\le m+n$;
    \item $f_{(1,-l)}=0\;\To\; f_{(2,-l)}=0$ for $0\le l<n$;
    \item $f_{(1,-l)}=0\;\To\; f_{(2,-l-n)}=0$ for $0\le l<m$.
  \end{itemize}
  Since $f_{(2,-m-n+1)}=0$, we get $f=0$.

  We have a projective resolution
  \[
    0\;\to\;\bigoplus_{l=m}^{2n-1}P_2(l)
    \;\to\;\bigoplus_{l=0}^{n-1}P_1(l)
    \;\to\; L_\lambda\;\to\; 0.
  \]
  A direct computation gives
  \[
    \Hom_{k\K_{0,-m,-n}^\op}\!\left(\bigoplus_{l=0}^{n-1}P_1(l),L_\mu\right)^{<0}
    =\bigoplus_{l=1}^{n-1}k(l)^{m-l}
    =\Hom_{k\K_{0,-m,-n}^\op}\!\left(\bigoplus_{l=m}^{2n-1}P_2(l),L_\mu\right)^{<0}.
  \]
  \Cref{lem:trivial,lem:i.e.} give the result.
\end{proof}

\begin{lem}\label{lem:special}
  Let $k\ge 2$ and $0<a_1\le a_2\le\cdots\le a_k$.  If $a_1+a_2\ge a_3$,
  then there is a pre-simple-minded collection
  $\{L_\lambda\}_{\lambda\in k^*}\subseteq\pvd k\K_{0,-a_1,-a_2,\ldots,-a_k}^\op$.
\end{lem}

\begin{proof}
  Let $L_\lambda\in\mod^{(-a_1-a_2,0]}k\K_{0,-a_1,-a_2}^\op
  \subseteq\mod^{(-a_1-a_2,0]}k\K_{0,-a_1,-a_2,\ldots,-a_k}^\op$
  be the module from \cref{prop:3-Kro}.  By assumption, the projective
  resolution of $L_\lambda$ takes the form
  \[
    0\;\to\;\bigoplus_{l=a_1}^{2a_2-1}P_2(l)\oplus P
    \;\to\;\bigoplus_{l=0}^{a_2-1}P_1(l)
    \;\to\; L_\lambda\;\to\; 0,
  \]
  where $P\in\add P_2(\ge a_1+a_2)$.  Since
  \[
    \Hom_{k\K_{0,-a_1,\ldots,-a_k}^\op}(P(\ge a_1+a_2),L_\mu)^{<0}
    =\Hom_{kQ^\op}(P_2,L_\mu)_{\le -a_1-a_2}=0,
  \]
  \cref{lem:trivial,lem:i.e.} give the result.
\end{proof}

\begin{prop}\label{prop:nKro}
  Let $k\ge 2$ and $0<a_1\le a_2\le\cdots\le a_k$.  Then there is a
  pre-simple-minded collection
  $\{L_\lambda\}_{\lambda\in k^*}\subseteq\pvd k\K_{0,-a_1,-a_2,\ldots,-a_k}^\op$.
\end{prop}

\begin{proof}
  \textsc{Case }$a_{k-2}-a_{k-3}\ge a_k-a_{k-1}$:
  By \cite[Lemma~6.1]{K94}, there is an equivalence
  \[
    (-)^!\;:=\;\RHom(-,S_1\oplus S_2)\colon
    \pvd k\K_{0,-a_1,\ldots,-a_k}^\op
    \;\simeq\;
    (\pvd k\K_{1,a_1+1,\ldots,a_k+1}^\op)^\op
  \]
  sending $S_i\mapsto P_i$.  Since
  $(a_k-a_{k-1})+(a_k-a_{k-2})\le a_k-a_{k-3}$,
  \cref{lem:special} produces a pre-simple-minded collection
  $\{L^!_\lambda\}_{\lambda\in k^*}\subseteq\pvd k\K_{1,a_1+1,\ldots,a_k+1}^\op$
  with
  $L^!_\lambda\in\add\Sigma^{(-a_k+a_{k-2},0]}P_1\ast
  \add\Sigma^{(-a_k-a_{k-1}+a_{k-2},0]}P_2$.
  Applying $(-)^!$ yields a pre-simple-minded collection
  $\{L_\lambda\}_{\lambda\in k^*}\subseteq\pvd k\K_{0,-a_1,\ldots,-a_k}^\op$
  with $L_\lambda\in\add\Sigma^{[0,a_k+a_{k-1}-a_{k-2})}S_2\ast
  \add\Sigma^{[0,a_k-a_{k-2})}S_1$.

  \medskip
  \textsc{Case }$a_{k-2}-a_{k-3}<a_k-a_{k-1}$:
  Let $l\ge 2$ be the largest index with $a_{l-2}-a_{l-3}\ge a_l-a_{l-1}$
  (setting $a_{-1}=-\infty$).  Since $a_{l+1}\ge a_l+a_{l-1}-a_{l-2}$,
  the argument above gives a pre-simple-minded collection
  $\{L_\lambda\}_{\lambda\in k^*}\subseteq\pvd k\K_{0,-a_1,\ldots,-a_k}^\op$
  with $L_\lambda\in\add\Sigma^{[0,a_l+a_{l-1}-a_{l-2}-1]}S_2\ast
  \add\Sigma^{[0,a_l-a_{l-2}-1]}S_1$.
\end{proof}

\bibliographystyle{alpha}
\bibliography{my}

\Addresses

\end{document}